\documentclass{article}

\usepackage{arxiv}

\usepackage[utf8]{inputenc} 
\usepackage[T1]{fontenc}    
\usepackage{hyperref}       
\usepackage{url}            
\usepackage{booktabs}       
\usepackage{amsfonts}       
\usepackage{nicefrac}       
\usepackage{microtype}      
\usepackage{graphicx}
\usepackage{amsmath}
\usepackage{amsthm}
\usepackage{cancel}
\newtheorem{thm}{Theorem}[section]
\newtheorem{lem}[thm]{Lemma}
\newtheorem{prop}[thm]{Proposition}
\newtheorem{cor}[thm]{Corollary}
\theoremstyle{definition}

\theoremstyle{remark}
\newtheorem{rem}{Remark}
\theoremstyle{definition}

\title{The Abundancy Index and Feebly Amicable Numbers}

\author{
  Jamie Bishop, Abigail Bozarth, Rebekah Kuss, Benjamin Peet Ph.D.\\
  Department of Mathematics\\
  St. Martin's University\\
  Lacey, WA 98503 \\
  \texttt{bpeet@stmartin.edu} \\
}

\begin{document}
\maketitle

\begin{abstract}
This research explores the sum of divisors - $\sigma(n)$ - and the abundancy index given by the function $\frac{\sigma(n)}{n}$. We give a generalization of amicable pairs - feebly amicable pairs (also known as harmonious pairs), that is $m,n$ such that $\frac{n}{\sigma(n)}+ \frac{m}{\sigma(m)}=1$. We first give some groundwork in introductory number theory, then the goal of the paper is to determine if all numbers are feebly amicable with at least one other number by using known results about the abundancy index. We establish that not all numbers are feebly amicable with at least one other number. We generate data using the R programming language and give some questions and conjectures.
\end{abstract}


\section{Introduction}

The sum of divisor function, $\sigma(n)$, for a positive integer $n$, is the sum of all the positive divisors including $n$ itself. Looking at the ratio of the sum of divisor function and the number itself, $\frac{\sigma(n)}{n}$, or the abundancy index, we look into its relation to concepts such as perfect numbers, abundant numbers, deficient numbers, and amicable numbers. Through understanding these relations, we will define the concept of feebly amicable numbers also known as harmonious numbers in \cite{kozek2015harmonious}. This is a generalization of an amicable number with weakened conditions so that the sum of divisors of the two do not need to be equal.

Formally, these are two numbers $m$ and $n$ such that $\frac{n}{\sigma (n)}+\frac{m}{\sigma (m)}=1$. Examples of the first twenty feebly amicable pairs are given for illustration of this concept.

We use the R programming language to produce some abundancy indices and then a list of feebly amicable numbers. This data allows us to ask some questions about feebly amicable numbers that are unknown about amicable numbers.

Our main new results are Theorem 8.1 and Corollary 8.2 which give conditions for when a number can be feebly amicable with another and consequently amicable. The final section has some questions and conjectures that might be of interest to the reader or for future work.

\section{History}

The implications of this research are derived from the historical mathematical workings of famous figures, most notably Euclid, Euler, and Mersenne. Euclid further advanced our understanding of prime numbers by providing the Euclidean algorithm and showing that there are infinitely many prime numbers. Euclid also completed one of the only proofs involving perfect numbers: if $2^{n}-1$ is prime, then $2^{n-1}(2^{n}-1)$ is perfect. Euler continued the idea of perfect numbers by proving that every even perfect number can be expressed in Euclid's form. This research also makes use of Marin Mersenne's work on primes, and a Mersenne prime is of the form $2^{n}-1$. \cite{stillwell1989mathematics}

\section{Preliminary Definitions}

We give some some preliminary definitions that we will rely upon throughout. We are working here with the positive integers (natural numbers). It should also be noted that it is possible to extend all this theory to the negative integers.

Firstly, a \textit{divisor} is a number that divides into another number without a remainder.Then, a \textit{prime} is a number that has only the divisors $1$ and itself. Particular prime numbers used in this paper are Mersenne primes which are prime numbers of the form $2^{p-1} (2^{p}-1)$, where $p$ is also prime. 

Given a natural number $n$, we can define the \textit{canonical representation} of $n$ to be $\prod_{i=1}^{r} p_{i}^{a_{i}}$ where the $p_{i}$ are the distinct prime divisors of $n$ and $a_{i}$ their multiplicities. 

Two number are \textit{coprime} (alternatively \textit{relatively prime}) if they share only $1$ as a divisor.

Then and importantly to this paper, the \textit{sum of divisor function} $\sigma(n)$ for a positive integer $n$ is defined as the sum of all its divisors (including $n$ itself).

From this, we can categorize natural numbers according to:

\begin{itemize}
\item $n$ is called \textit{perfect} if $\sigma(n)=2n$.
\item $n$ is called \textit{abundant} if $\sigma(n)>2n$.
\item  $n$ is called \textit{deficient} if $\sigma(n)<2n$.
\end{itemize} 

This paper generalizes the definition of amicable numbers. To be precise, \textit{amicable numbers} are two numbers related in such a way that the sum of the proper divisors of each are equal and also equal to the sum of the numbers.  That is, $m$,$n$ are amicable if $m+n=\sigma(m)=\sigma(n)$.

Part of this paper is an investigation of the \textit{abundancy index} of a number. It is defined by $\lambda(n)=\frac{\sigma(n)}{n}$ and, in some sense, measures how divisible a number is.

\textit{Multiply-perfect numbers} are numbers such that their abundancy index is an integer. Note that perfect numbers by definition have abundancy index $2$.

Finally, two numbers in $\mathbb{N}$ are \textit{friendly} if they have the same abundancy index. That is, $\frac{\sigma(n)}{n}=\frac{\sigma(m)}{m}$. More generally, friendly numbers form \textit{friendly clubs} if they all have the same abundancy index.

\section{Preliminary Results}

We now present some preliminary results that give some illustration of the theory and that will be used throughout the rest of the paper. Good references for these and more foundational theory are \cite{long1987elementary} and \cite{dudley2012elementary}. However, there are many good texts on introductory number theory.

\begin{prop} If $p$,$q$ are distinct primes then $\sigma(p^{n}q^{m}) = \sigma(p^{n})\sigma(q^{m})$. That is, the multiplicative property can be applied where $\sigma(mn)=\sigma(m)\sigma(n)$ when $gcd(mn)=1$.
\end{prop}

\begin{proof}
First let $p$ and $q$ be primes, then:
\begin{center}
$\sigma(p^{n}q^{m})=\sum\limits_{i=0}^{n}\sum\limits_{j=0}^{m}p^{i}q^{j}=(1+p+\ldots+p^{n})(1+q+\ldots+q^{m})=\sigma(p^n)\sigma(q^m)$
\end{center}
\end{proof}

The above Proposition 4.1 yields the multiplicative property of the sum of divisors function. That is, if $m$ and $n$ are coprime, then $\sigma(mn)=\sigma(m)\sigma(n)$.

This leads to the following which is Theorem 2.24 of \cite{long1987elementary}:

\begin{thm}If $a=\prod_{i=1}^{r} p_i^{a_{i}}$ where $a_{i}>0$ for each $i$ is the canonical representation of $a$, then $$\sigma(a)=\prod_{i=1}^{r} \frac{p_{i}^{a_{i}} -1}{p_{i} -1}$$.
\end{thm}

\begin{proof}
We first establish that:
$$\sigma(p^n)=1+p+p^2+p^3+...+p^n=\frac{p^{n+1}-1}{p-1}$$

This follows as:

$$(p-1)(1+p+p^2+p^3+...+p^n) =p^{n+1}-1$$

Then by applying Proposition 4.1 inductively, the result follows.
\end{proof}

We make here a number of remarks regarding values of the abundancy index and how it relates to perfect, abundant, deficient, and friendly numbers.

\begin{rem} The codomain of the abundancy index, $\lambda(n)=\frac{\sigma(n)}{n}$, is $\mathbb{Q} \cap (1,\infty)$.
\end{rem}
 We will see later that this codomain is not in fact the range, that is, there are values in $\mathbb{Q} \cap (1,\infty)$ that are not abundnacy indices.

\begin{rem}When $\frac{\sigma(n)}{n}=2$, then $n$ is perfect. When $1 < \frac{\sigma(n)}{n} < 2$, then $n$ is deficient.  And when $\frac{\sigma(n)}{n}>2$, then $n$ is abundant.  Additionally, all perfect numbers in this case are friendly to one another.  That is $\frac{\sigma(n)}{n}=\frac{\sigma(m)}{m}$; $\lambda(n) = \lambda(m)$. 
\end{rem}

We give a proof of the Euclid-Euler Theorem to illustrate the theory:

\begin{thm}
An even number is perfect if and only if it has the form $2^{p-1}(2^{p}-1)$, where $2^{p}-1$ is prime.
\end{thm}

\begin{proof}
Let $2^{p}-1$ be prime. Then, by the multiplicative property, the sum of divisors of $2^{p-1}(2^{p}-1)$ is equal to  $$\sigma(2^{p-1}(2^{p}-1))=\sigma(2^{p-1})\sigma(2^{p}-1)=(2^{p}-1)2^{p}=2(2^{p-1})(2^{p}-1).$$
Hence, since the sum of the divisors of $2^{p-1}(2^{p}-1)$ is twice itself, $2^{p-1}(2^{p}-1)$ is perfect.

For the converse, let $2^{k}x$ be an even perfect number, where $x$ is odd. For $2^{k}x$ to be a perfect number, the sum of its divisors must be twice its value. So, $$2(2^{k}x)=\sigma(2^{k}x)=\sigma(2^{k})\sigma(x)=(2^{k+1}-1)\sigma(x)$$ by the multiplicative property of $\sigma$.

The factor $2^{k+1}-1$ must divide $x$. So $y=x/(2^{k+1}-1)$ is a divisor of $x$. Now,
$$2^{k+1}y=\sigma(x)=x+y+z=2^{k+1}y+z,$$ where $z$ is the sum of the other divisors.
Thus, for this equality to be true, there must be no other divisors, so $z$ must be 0. Hence, $y$ must be 1, and $x$ must be a prime of the form $2^{k+1}-1$. Therefore, an even number is perfect if and only if it has the form $2^{p-1}(2^{p}-1)$, where $2^{p}-1$ is prime.
\end{proof}

We also give the following proposition that we will utilize later on:

\begin{prop} If $p$ is prime, then $\frac{\sigma(n)}{n}=\frac{p+1}{p}$ if and only if $n=p$.
\end{prop}

\begin{proof}

We first show that if $p$ is prime and $\frac{\sigma(n)}{n}=\frac{p+1}{p}$, then $n=p$.
So we have $\frac{\sigma(n)}{n}=\frac{(p+1)k}{(p)k}$   for some $k=p^a L$ where $p$ does not divide $L$ and $n=pk$. Then suppose for contradiction that $k>1$.

Then $\sigma(n)=(p+1)k$, but also, $\sigma(n)=\sigma(pk)=\sigma(pp^a L)$. By the multiplicative property of $\sigma$ and Theorem 4.2, $\sigma(pp^a L)=\sigma(p^{a+1}) \sigma(L)=\frac{p^{a+2}-1}{p-1} \sigma(L)$. So, $(p+1)k=\frac{p^{a+2}-1}{p-1} \sigma(L)$.

We continue the calculation with $(p+1)(p^a L)=\frac{p^{a+2}-1}{p-1} \sigma(L)$.

Hence, $(p+1)(p-1)(p^a L)=(p^{a+2}-1) \sigma(L)$, which gives, $(p^2 -1)(p^a L)=(p^{a+2}-1) \sigma(L)$, and finally, $(p^{a+2} - p^a)L=(p^{a+2}-1) \sigma(L)$.

Because $(p^{a+2} - p^a)<(p^{a+2}-1)$ and $L<\sigma(L)$ we must have that $(p^{a+2} - p^a)L<(p^{a+2}-1)\sigma(L)$. So, by contradiction, $n=p$.

For the converse if $n=p$, then if $p$ is prime, as $\sigma(p)=p+1$ we must have:

\begin{center}
    $\frac{\sigma(n)}{n}=\frac{\sigma(p)}{p}=\frac{p+1}{p}$.
\end{center}
\end{proof}

\section{R code}

In order to explore the values of both the sum of divisors function and the abundancy index, we used some code in the R programming language \cite{R} to generate the first 100,000 values of the sum of divisor function and abundancy indices. The following is some code that describes the algorithm:

\begin{verbatim}
    sod <- function(x) %This defines the sum of divisors function%
        {s<-0
        for(i in 1:x){if(x%%i==0){s<-s+i}} %This loops through the values 1 through x to see 
        which are factors and adds them to the sum if they are%
        return(s)}

    sigma<-c(1:100000) %This defines a vector of length 100,000%

    for(i in 1:100000){sigma[i]<-sod(i)} %This gives a vector of the first 100,000 values of  
    the sum of divisors%

    abun<-c(1:100000) %This again defines a vector of length 100,000%

    for(i in 1:100000){abun[i]<-sigma[i]/i} %This gives a vector of the first 100,000 values 
    of the abundancy index%
\end{verbatim}

Figure 1 shows a histogram to reflect the data.

\begin{figure}[ht]
\centering
\includegraphics[height=10cm]{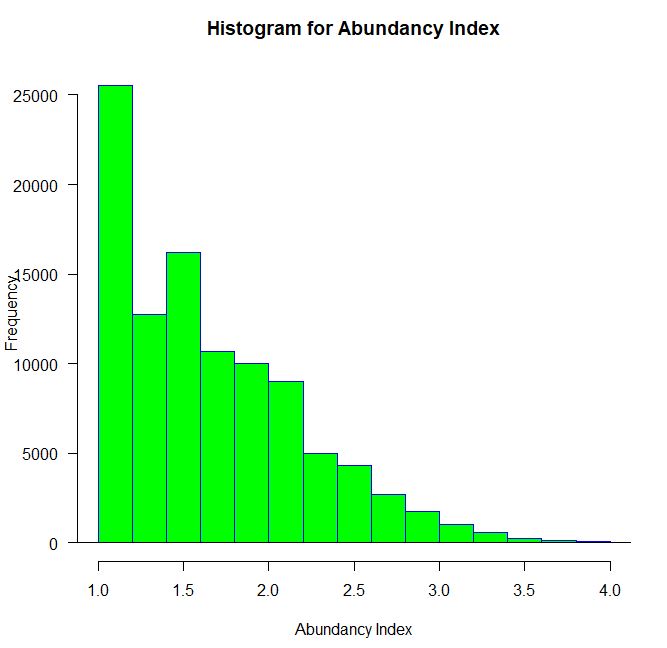}
\caption{Histogram for Abundancy Index}
\end{figure}

By our code we calculated the fraction of abundant numbers in the first $100,000$ numbers to be $0.24799$. This is not within the range given by \cite{kobayashi2010density}, but it is close. In that paper, they compute that $\alpha=\displaystyle\lim_{n\rightarrow \infty} \frac{A(n)}{n}$ is such that 

$$0.2476171<\alpha<0.2476475$$

Here $A(n)$ is the number of abundant numbers less than $n$.

We suggest that computing larger numbers of abundancy indices would give an estimate within the proven bounds. In any case, the histogram illustrates how there are roughly three times more deficient numbers than abundant numbers.

\section{Range of the Abundancy Index}

From our data, $\frac{5}{4}$ does not appear as an abundancy index of any $n$ less than $100,000$. First, we prove that $\frac{5}{4}$ does not appear for any natural number, and then, we show a generalization for finding more numbers not in the range.

\begin{lem}
$\frac{5}{4} \neq \frac{\sigma(n)}{n}$ for any natural number $n$.
\end{lem}

\begin{proof}
Suppose to the contrary. So $5k=\sigma(n)$ and $4k=n$ for some $k$.  Thus, $n=4k=2^{a+2}l$ for some nonnegative integer $a\in\mathbb{N}$ and odd integer $l\in\mathbb{N}$. So then  $\sigma(n)=\sigma(2^{a+2}l)$, which by the multiplicative property of $\sigma$, $\sigma(2^{a+2}l)=\sigma(2^{a+2})\sigma(l)=(2^{a+3}-1)\sigma(l)$. 

This leads to: $$\frac{\sigma(n)}{n}=\frac{(2^{a+3}-1)\sigma(l)}{2^{a+2}l}>\frac{7}{4}\frac{\sigma(l)}{l}>\frac{7}{4}.$$

Therefore, $\frac{5}{4} \neq \frac{\sigma(n)}{n}$.
\end{proof}

This can be generalized to Theorem 1 from \cite{weiner2000abundancy}:

\begin{thm}
If $k$ is coprime to $m$, and $m<k<\sigma(m)$, then $\frac{k}{m}$ is not the abundancy index of any integer.
\end{thm}
\begin{proof}
Assume $\frac{k}{m}=\frac{\sigma(n)}{n}$. Then $m\sigma(n)=kn$, so $m|kn$, hence $m|n$ because $(k,m)=1$. But because $m|n$ implies $\frac{\sigma(m)}{m} \leq \frac{\sigma(n)}{n}$, with equality only if $m=n$, $\frac{\sigma(m)}{m} \leq \frac{\sigma(n)}{n}=\frac{k}{m}$, contradicting the assumption $k<\sigma(m)$.
\end{proof}

We can see that not all rational numbers greater than $1$ are in the range, but we ask the question: is the range dense in rationals greater than 1? Recall that if $S \subseteq T \subseteq \mathbb{R}$, we say that $S$ is dense in $T$ if for any two numbers in $T$ there exists an element of $S$ in between them.

We refer to \cite{laatsch1986measuring} and \cite{weiner2000abundancy} again for the following results:

\begin{thm} (Theorem 5 in \cite{laatsch1986measuring}) The set$\{\lambda(n)|n\in\mathbb{N}\}$ is dense in $\mathbb{Q}\cap(1,\infty)$.
\end{thm}

We give the proofs from \cite{weiner2000abundancy} of the following results for exposition:

\begin{lem}
Let $m$ be a positive integer. If $p$ is prime with $p>2m$, then among any $2m$ consecutive integers, there is at least one integer coprime with $pm$.
\end{lem}

\begin{proof}
Let $S$ be any set of $2m$ consecutive integers. If $p>2m$ there is at most one multiple of $p$ in $S$. But $S$ contains at least two integers coprime with $m$, one of which is coprime with $p$ and, therefore, also $pm$.
\end{proof}

\begin{thm} (Theorem 2 in \cite{weiner2000abundancy}) The complement of $\{\lambda(n)|n\in\mathbb{N}\}$ in $\mathbb{Q}\cap(1,\infty)$ is dense in $\mathbb{Q}\cap(1,\infty)$.
\end{thm}
\begin{proof}
Choose any real $x > 1$, and any $\epsilon>0$. We will exhibit a rational in the interval $(x-\epsilon , x+\epsilon)$, and that is not an abundancy ratio. By Theorem 6.3, choose $m>1$ so that the abundancy index $\frac{\sigma(m)}{m}$ is in the interval $(x-\frac{\epsilon}{2}, x+\frac{\epsilon}{2})$. For every prime $p>2m$, we have:

\begin{center}
    $x-\frac{\epsilon}{2}<\frac{\sigma(m)}{m}<\frac{\sigma(pm)}{pm}=(1+\frac{1}{p})\frac{\sigma(m)}{m}<(1+\frac{1}{p})(x+\frac{\epsilon}{2})$.
\end{center}

If we also require $p>\frac{2x+\epsilon}{\epsilon}$, then $(1+\frac{1}{p})(x+\frac{\epsilon}{2})<x+\epsilon$, we have: 

\begin{center}
    $x-\frac{\epsilon}{2}<\frac{\sigma(pm)}{pm}<x+\epsilon$.
\end{center}

By the Lemma 6.4, we know that $\sigma(pm)-k$ is coprime with $pm$ for some $k$ with $1 \leq k \leq 2m$. For such $k$, we also have:

\begin{center}
    $\sigma(pm)-k \geq \sigma(pm)-2m \geq (p+1)(m+1)-2m>pm$
\end{center}

because $p>2m$. Therefore, by Theorem 6.2, $\frac{\sigma(pm)-k}{pm}$ is not an abundancy index. So, then:

\begin{center}
    $\frac{\sigma(pm)-k}{pm} \geq \frac{\sigma(pm)-2m}{pm}= \frac{\sigma(pm)}{pm}-\frac{2}{p}>x-\frac{\epsilon}{2}-\frac{2}{p}$.
\end{center}

If $p\geq \frac{4}{\epsilon}$, we have $x-\frac{\epsilon}{2}-\frac{2}{p} \geq x-\epsilon$, thus $\frac{\sigma(pm)-k}{pm}>x-\epsilon$. All the inequalities are satisfied if $p>max \{ 2m,\frac{2x+\epsilon}{\epsilon}, \frac{4}{\epsilon} \}$, and so:

\begin{center}
    $x-\epsilon<\frac{\sigma(pm)-k}{pm} < \frac{\sigma(pm)}{pm}<x+\epsilon$.
\end{center}
 This gives, $\frac{\sigma(pm)-k}{pm}$ that is not an abundancy index, within $\epsilon$ of $x$.
\end{proof}

\section{Feebly Amicable Numbers}

We now proceed to generalize the definition of amicable numbers. We do so by recognizing the following result:

\begin{prop}
If two numbers $m$, $n$ are amicable, then $\frac{n}{\sigma(n)}+\frac{m}{\sigma(m)}=1$.
\end{prop}
\begin{proof}
Let $m$ and $n$ be amicable numbers. Then $\sigma(m)=\sigma(n)=m+n$ and so $\frac{m+n}{\sigma(n)}=1$ hence
$\frac{n}{\sigma(n)}+\frac{m}{\sigma(n)}=1$.

Thus, if $m$ and $n$ are amicable numbers, then  $\frac{m}{\sigma(m)}+\frac{n}{\sigma(n)}=1$.
\end{proof}

This allows us to formulate the following definition:

\textit{Feebly amicable numbers} are pairs $m$, $n$ such that $$\frac{n}{\sigma (n)}+ \frac{m}{\sigma(m)}=1$$

Alternatively, we can define in terms of the abundancy index $\lambda$:

$$\frac{1}{\lambda(n)}+\frac{1}{\lambda(m)}=1$$

In words, two numbers are feebly amicable if the reciprocals of their abundancy indices sum to 1. Note that feebly amicable pairs are referred to as harmonious pairs in \cite{kozek2015harmonious}.

\begin{rem} We note that pairs of perfect numbers are not amicable (no two perfect numbers have the same sum of divisors). However, they are feebly amicable, and the Venn diagram in Figure 2 illustrates the containment shown in Proposition 7.1.
\end{rem}

We also note that amicable pairs have been extended to \textit{amicable triples} which are three numbers $m$, $n$, and $s$ such that $\sigma (m) =\sigma(n)=\sigma(s)=m+n+s$, and therefore the following is true as well: $\frac{n}{\sigma(n)} + \frac{m}{\sigma(m)} + \frac{s}{\sigma(s)} =1$.

If amicable pairs and triples have been defined such as they have been above, then \textit{amicable k-tuples} are numbers $n_{1}, \ldots, n_{k}$ such that $$\sigma(n_{1})=\ldots=\sigma(n_{k})=n_{1}+\ldots+n_{k}$$

Given these definitions we can also generalize to \textit{feebly amicable triples}. These are three numbers $m$, $n$, and $s$ such that $$\frac{n}{\sigma(n)}+ \frac{m}{\sigma(m)}+\frac {s}{\sigma(s)}=1$$ and \textit{feebly amicable k-tuples} as numbers $n_{1}, \ldots, n_{k}$ such that $$\frac{n_{1}}{\sigma(n_{1})}+\ldots+\frac{n_{k}}{\sigma (n_{k})}=1.$$

\begin{figure}[ht]
\centering
\includegraphics[height=6cm]{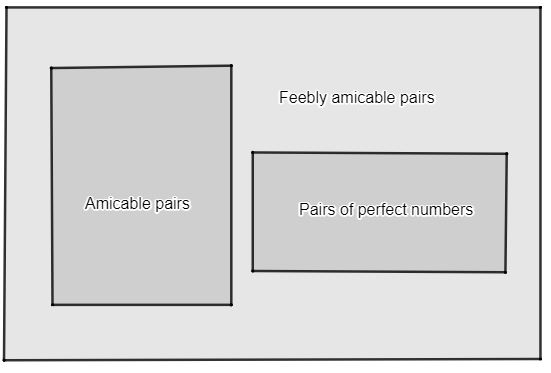}
\caption{Venn diagram}
\end{figure}

\begin{rem} All members of a friendly club have the same abundancy index. It is possible to talk of \textit{feebly amicable clubs}. As if $m,n$ are feebly amicable then $m,l$ are feebly amicable for any $l$ in the same friendly club as $n$.
\end{rem}

\section{New Results}

We now ask the following question: are all integers feebly amicable with some other integer? To see that this is not true, we establish the following:

\begin{thm}
Let $k$ and $m$ be such that $k$ is coprime with $m$ and $m<k<\sigma(m)$. If some $n$ has abundancy index $\frac{k}{k-m}$, then $n$ is not feebly amicable with any other integer.
\end{thm}

\begin{proof}
Suppose that $\lambda(n)=\frac{\sigma(n)}{n}=\frac{k}{k-m}$.  Since we already know $\frac{1}{\lambda(n)}+\frac{1}{\lambda(m)}=1$, then we can substitute in and get:
$$\frac{k-m}{k}+\frac{1}{\lambda(m)}=1$$.
$$\frac{1}{\lambda(m)}=1-\frac{k-m}{k}$$
$$\hspace{1.3cm}=\frac{k-(k-m)}{k}$$
$$\hspace{-0.2cm}=\frac{m}{k}$$.

However, there is no such abundancy index as $\lambda(m)=\frac{k}{m}$, by Theorem 6.2.  Thus, proving the theorem.
\end{proof}

To see that this is not vacuous note that we have seen that $\frac{5}{4}$ is not an abundancy index, yet the abundancy index of $14182439040$ is known to be $5=\frac{5}{5-4}$ as it is a multiply perfect number of order 5. See OEIS A007539 \cite{oeis}.

A natural corollary of this theorem is the following:

\begin{cor}
If $k$ is coprime with $m$, $m<k<\sigma(m)$, and $n$ has abundancy index $\frac{k}{k-m}$, then $n$ is not amicable with any other integer.
\end{cor}

This follows directly from the fact that amicable pairs are feebly amicable pairs. Hence, in particular $14182439040$ has no amicable pair.

Proposition 4.4 showed that the only number with abundancy index $\frac{p+1}{p}$ was $p$ where $p$ was prime. The question arises, is $p$ feebly amicable with any number? Naturally, such a number would need abundancy index $p+1$. That is, $p+1$-perfect. We state this as a result:

\begin{prop}
For $p$ prime, $p$ is feebly amicable with $n$ if and only if $n$ is $(p+1)$-perfect.
\end{prop}

It is unknown if there are any coprime amicable pairs \cite{Garcia}. Hence a natural question is whether there are any coprime feebly amicable pairs. The data reveals that there are none less than $1,000$ but that the first coprime feebly amicable pair is $1485$ and $868$. There are another four pairs before $5,000$.

\section{Examples}

We now generate the first 20 feebly amicable pairs that are not amicable or pairs of perfect numbers. To do so, we implemented the following code:

\begin{verbatim}
    for(i in 1:100000){for(j in 1:i){if(1/abun[i]+1/abun[j]==1){print(i) print(j)}}}
\end{verbatim}

These pairs are consistent with those listed in \cite{oeis} in A253534 and A253535.

\begin{center}
\begin{tabular}{ |c c| } 
 \hline
 12 & 4 \\ 
  \hline
 30 & 14 \\ 
  \hline
 40 & 10 \\ 
  \hline
  44 & 20 \\
   \hline
   56 & 8 \\
    \hline
    84 & 15 \\
     \hline
 96 & 26 \\
  \hline
  117 & 60 \\
   \hline
   120 & 2 \\
    \hline
  135 & 42 \\
   \hline
 140 & 14 \\
  \hline
 182 & 66 \\
  \hline
  184 & 88 \\
   \hline
   190 & 102 \\
    \hline
    198 & 45 \\
     \hline
 224 & 10 \\
  \hline
 234 & 4 \\
  \hline
  248 & 174 \\
   \hline
   252 & 153 \\
    \hline
      260 & 164 \\
   
 \hline
\end{tabular}
\end{center}

\section{Questions and Conjectures}

This section considers some questions and conjectures that we have encountered in the course of writing this paper.

\textit{Question 1} Are there infinitely many coprime feebly amicable pairs?

Given the regularity of coprime pairs of feebly amicable pairs, we conjecture that there are an infinite number.

In \cite{erdos}, Paul Erdős proved that the asymptotic density of amicable integers relative to the positive integers was $0$. That is, the ratio of the number of amicable numbers less than $n$ with $n$ tends to zero as $n$ tends to infinity. This gives rise to the question:

\textit{Question 2} What is the density of feebly amicable numbers relative to the positive integers?

Given the number of feebly amicable numbers in the first $5,000$ integers is $310$, then $178$ in the next $5,000$, and $136$ in the third $5,000$, it appears that the density is decreasing and so the asymptotic density is at least less than $0.0272=\frac{136}{15,000}$. Indeed, \cite{kozek2015harmonious} gives an upper bound and confirms that the asymptotic density is $0$.

The sum of amicable numbers conjecture \cite{guy2004unsolved} states that as the largest number in an amicable pair approaches infinity, the percentage of the sums of the amicable pairs divisible by ten approaches $100\%$. We therefore ask the question:

\textit{Question 3} As the larger number in a feebly amicable pair approaches infinity, does the percentage of the sums of the pairs divisible by ten approach $100\%$?

Our data tells us that in the first $5000$ numbers there are 11 feebly amicable pairs that sum to a multiple of 10. There are then 8 in the next $5000$ and 4 between $10000$ and $15000$. As a sequence of fractions of the number of feebly amicable pairs this is: $0.035, 0.045, 0.029$. This does not give us any noticeable trend and once again much higher values would be required to make a strong conjecture other than to say it does not appear that the percentage tends towards 100$\%$ as in the amicable case. Computing further is certainly possible, but it becomes much more computationally expensive to compute sum of divisors and hence abundancy indices. \cite{sladkey2012successive} provides an algorithm to compute in $O(n^{\frac{1}{3}})$ time, but our code was much more crude.

\bibliographystyle{unsrt}  
\bibliography{references} 

\end{document}